\documentclass[12pt,a4paper,reqno]{amsart}
\usepackage{hyperref}
\usepackage{cleveref}
\usepackage{amssymb}
\usepackage{verbatim}
\usepackage{amscd}
\usepackage{enumerate}
\usepackage{graphicx}
\usepackage{siunitx}
\usepackage{tikz-cd}
\usepackage{stix}
\usepackage{bm}
\usepackage{mathtools}
\usepackage[tableposition=top]{caption}
\usepackage{booktabs,dcolumn}
\usepackage{tikz}
\usetikzlibrary{shapes.geometric, arrows}

\numberwithin{equation}{section}

\theoremstyle{plain}

\newtheorem{theorem}{Theorem}[section]
\newtheorem{proposition}[theorem]{Proposition}
\newtheorem{lemma}[theorem]{Lemma}
\newtheorem{corollary}[theorem]{Corollary}

\theoremstyle{definition}

\newtheorem{remark}[theorem]{Remark}

\newcommand\N{\mathbb{N}}

\newcommand\eps{\varepsilon}

\begin{document}

\title[Monotone sequences of totient function]{Monotone nondecreasing sequences of the Euler totient function}

\author{Terence Tao}
\address{UCLA Department of Mathematics, Los Angeles, CA 90095-1555.}
\email{tao@math.ucla.edu}


\begin{abstract}  Let $M(x)$ denote the largest cardinality of a subset of $\{n \in \N: n \leq x\}$ on which the Euler totient function $\varphi(n)$ is nondecreasing.  We show that $M(x) = (1+O(\frac{(\log\log x)^5}{\log x})) \pi(x)$ for all $x \geq 10$, answering questions of Erd\H{o}s and Pollack--Pomerance--Trevi\~no. A similar result is also obtained for the sum of divisors function $\sigma(n)$.
\end{abstract}

\maketitle


\section{Introduction}

Given any finite set $A$ of natural numbers, let $M(A)$ denote the largest cardinality of a subset of $A$ on which the Euler totient function $\varphi(n)$ is nondecreasing (or weakly increasing), thus $\varphi(n) \leq \varphi(m)$ whenever $n \leq m$ are both in $A$.  We also write $M(x)$ for $M([x])$, where $[x]$ denotes the set $[x] \coloneqq \{n \in \N: n \leq x \}$.  For instance, $M(6)=5$, since $\varphi$ is nondecreasing on $\{1,2,3,4,5\}$ or $\{1,2,3,4,6\}$ but not $\{1,2,3,4,5,6\}$. The first few values of $M(x)$ are
$$ 1, 2, 3, 4, 5, 5, 6, 6, 7, 7, 8, 8, 9, 9, 10, 11, 12, 12, \dots$$
(\url{https://oeis.org/A365339}).
Since $\varphi(p)=p-1$ for primes $p$, we clearly have\footnote{See Section \ref{notation-sec} for our asymptotic notation conventions.}
\begin{equation}\label{mlow}
M(x) \geq \pi(x) = \left(1+O\left(\frac{1}{\log x}\right) \right) \frac{x}{\log x}
\end{equation}
for $x \geq 2$, by the prime number theorem.  As is customary, we use $\pi(x)$ to denote the number of primes in $[x]$.  In the opposite direction, it was shown in \cite{pomerance} that
$$ M(x) \leq (1-c+o(1)) W(x)$$
as $x \to \infty$ for some absolute constant $c > 0$, where
$$ W(x) \coloneqq \# \{ \varphi(n): n \in [x] \}$$
(\url{https://oeis.org/A061070}) is the range of values of the Euler totient function up to $x$, and $\# A$ denotes the cardinality of a set $A$.  The precise order of growth of $W$ is known \cite{ford}; applying this estimate (or the earlier bounds from \cite{maier}), we obtain the upper bound
$$ M(x) \leq \exp( (C+o(1)) \log_3^2 x ) \frac{x}{\log x}$$
as $x \to \infty$ for an explicit constant $C = 0.81781\dots$, where we use the abbreviations
$$ \log_2 x \coloneqq \log\log x; \quad \log_3 x \coloneqq \log\log\log x.$$

The main result of this paper is to improve the upper bound on $M(x)$ to asymptotically match the lower bound, with only a slight degradation in the error term.

\begin{theorem}[New upper bound]\label{thm:main}  We have $M(x) \leq \left(1+O\left(\frac{\log^5_2 x}{\log x}\right)\right) \frac{x}{\log x}$ for all $x \geq 10$.  In particular, by \eqref{mlow}, we have
\begin{equation}\label{M-pi}
\pi(x) \leq M(x) \leq \left(1+O\left(\frac{\log^5_2 x}{\log x}\right)\right) \pi(x)
\end{equation}
or equivalently
$$M(x) = \left(1+O\left(\frac{\log^5_2 x}{\log x}\right)\right) \frac{x}{\log x}.$$
\end{theorem}

This answers a question of Erd\H{o}s \cite[\S 9]{erdos}, who asked if $M(x) = (1+o(1)) \pi(x)$.  As corollary, we may also answer a ``not unattackable'' question from \cite[p. 381]{pomerance}:

\begin{corollary}  If ${\mathcal I}$ is a subset of $[x]$ on which $\varphi$ is nondecreasing, then $\sum_{n \in {\mathcal I}} \frac{1}{n} \leq \log_2 x + O(1)$ for $x \geq 10$.
\end{corollary}

\begin{proof}  We may assume $x$ is an integer. By summation by parts (or telescoping series), we have
\begin{align*}
\sum_{n \in {\mathcal I}} \frac{1}{n} &= \sum_{m \leq x} \frac{1}{m(m+1)} \sum_{n \in {\mathcal I}: n \leq m} 1 + \frac{1}{x+1} \sum_{n \in {\mathcal I}} 1 \\
&\leq 1 + \sum_{m \leq x} \frac{1}{m^2} \sum_{n \in {\mathcal I}: n \leq m} 1.
\end{align*}
By Theorem \ref{thm:main}, we have
$$ \sum_{n \in {\mathcal I}: n \leq m} 1 \leq \left(1+O\left(\frac{\log^5_2 m}{\log m}\right)\right) \frac{m}{\log m}$$
for all large $m$.  The claim then follows from the absolute convergence of $\sum_{m \geq 10} \frac{\log^5_2 m}{m \log^2 m}$, and the standard asymptotic $\sum_{10\leq m \leq x} \frac{1}{m \log m} = \log_2 x + O(1)$ for $x \geq 100$.
\end{proof}

In \cite{pomerance} the much stronger claim
\begin{equation}\label{mpi}
 M(x) \leq \pi(x) + O(1)
\end{equation}
was posed as an open question, with numerics in fact suggesting the even more precise conjecture
\begin{equation}\label{pi-64}
    M(x)=\pi(x) + 64
\end{equation}
for $x \geq 31957$ (see \cite[\S 9]{pomerance}, where the conjecture was verified\footnote{This verification was recently extended to $m=8,9$ by Chai Wah Wu; see \url{https://oeis.org/A365474}.} for $x = 10^m$ for $m \leq 7$).  Unfortunately, there appear to be several obstacles to establishing such a strong result, which we discuss in Section \ref{obstruct}. Nevertheless, some improvements in the error term in \eqref{M-pi} may be possible, particularly if one is willing to establish results conditional on such conjectures as Cram\'er's conjecture \cite{cramer} or quantitative forms of the Dickson--Hardy--Littlewood prime tuples conjecture \cite{dickson}, \cite{hardy}; again, see \Cref{obstruct}.

\begin{remark} The lengths $M^\downarrow(x)$ (resp. $M^0(x)$) of the longest sequences in $[x]$ on which $\varphi$ is \emph{nonincreasing} (resp. \emph{constant}), rather than nondecreasing, are considerably smaller than $M(x)$.  Indeed, in \cite{pomerance} the bounds
$$ M^0(x) + x^{0.18} < M^\downarrow(x) \ll x \exp\left(-\left(\frac{1}{2}+o(1)\right) \log^{1/2} x \log^{1/2}_2 x \right)
$$
were shown for large $x$, with the speculation that ``perhaps'' $M^\downarrow(x) \asymp M^0(x)$, while it is also known (see \cite{pom-popular}, \cite[\S 4]{pom-two} for the upper bound, and \cite{lichtman} for the lower bound, which builds upon earlier work in \cite{bh}, \cite{harman}) that
\[
    x^{0.7156} \ll M^0(x) \ll x \exp(-(1+o(1)) \log x \log_3 x / \log_2 x)
\]
for large $x$, with the upper bound conjectured to be sharp.  See \cite{pomerance} for further discussion.
\end{remark}

The arguments in \cite{pomerance} were largely based on an analysis of equations of the form $\varphi(n)=\varphi(n+k)$.  To get the more precise asymptotics of Theorem \ref{thm:main}, we proceed differently, as we shall now explain. Perhaps surprisingly, we do not use deep results from analytic number theory; aside from the prime number theorem with classical error term (\Cref{prime}), all of our tools are elementary.

It is first necessary to establish the following sharp inequality regarding the (weighted) multiplicity of the function $d \mapsto \varphi(d)/d$.

\begin{proposition}[Preliminary inequality]\label{prelim}  For any positive rational number $q$, one has
$$ \sum_{d: \frac{\varphi(d)}{d} = q} \frac{1}{d} \leq 1.$$
\end{proposition}

The reason we need to establish this inequality is as follows.  Suppose for contradiction that the inequality failed, then one could find a rational number $q$ and a finite set ${\mathcal D}$ of natural numbers $d$ with $\frac{\varphi(d)}{d}=q$ such that
$$ \sum_{d \in {\mathcal D}} \frac{1}{d} > 1.$$
Observe that if $n = dp$ for $d \in {\mathcal D}$ and $p$ a prime larger than the largest element $\max {\mathcal D}$ of ${\mathcal D}$, then
\begin{equation}\label{varphi-form}
    \varphi(n) = \varphi(d) \varphi(p) = qd (p-1) = q (n-d).
\end{equation}
To exploit this, let $A$ be the set of numbers in $[x]$ of the form $dp$, where $d \in {\mathcal D}$ and $p > \max \mathcal{D}$ is a prime, and let $A'$ be the set $A$ with any consecutive elements of $A$ with difference less than $\max {\mathcal D}$ deleted.  If $d_1p_1 < d_2 p_2$ are consecutive elements of $A'$, then by \eqref{varphi-form} we have
$$ \varphi(p_1 d_1) - \varphi(p_2 d_2) = q ( p_1 d_1 - d_1 - p_2 d_2 + d_2) \geq q (\max \mathcal{D} - d_1 + d_2) \geq q d_2 > 0.$$
In particular,  $\varphi$ is increasing on $A'$, and one can check using the prime number theorem (as well as standard upper bound sieves to bound the number of elements deleted) that
$$ \# A' = \left(\sum_{d \in {\mathcal D}} \frac{1}{d} + o(1)\right) \frac{x}{\log x}$$
as $x \to \infty$, which would contradict Theorem \ref{thm:main}.  Thus we must prove Proposition \ref{prelim} first.  Fortunately, this is easy to accomplish thanks to a minor miracle: the sum appearing in Proposition \ref{prelim} can be computed almost exactly whenever it is non-zero!  See Lemma \ref{exact} below.

We now informally outline the proof of Theorem \ref{thm:main}. Standard results from the anatomy of integers will tell us that all numbers $n$ in $[x]$ outside of a small exceptional set $E$ will be of one of the following two forms:
\begin{itemize}
    \item[($A_1$)] $n=dp$, where $d$ is a very smooth number (all prime factors small), and $p$ is a large prime.
    \item[($A_2$)] $n=dps$, where $p$ is a medium-sized prime, $d$ has all prime factors less than $p$, and $s$ is an almost prime (a prime or product of two primes), with all prime factors of $s$ much larger than $p$.
\end{itemize}
The precise definitions of $A_1$, $A_2$ and $E$ are given in Section \ref{main-sec}, with the decomposition justified by Lemma \ref{lem-decomp} and Proposition \ref{prop-e}.

Suppose first that $n=dps$ is of type $A_2$. From multiplicativity we have the identity
$$ \varphi(n) = \frac{\varphi(d)}{d} \left(1 - \frac{1}{p}\right) \frac{\varphi(s)}{s} n;$$
since $s$ is almost prime and has large prime factors we are led to the approximation
$$ \varphi(n) \approx \frac{\varphi(d)}{d} \left(1 - \frac{1}{p}\right) n.$$
The upshot of this approximation is that if we restrict $d$ to a fixed value, and restrict $n$ to a somewhat short interval, then $\varphi(n)$ can only be nondecreasing if the prime $p$ is also essentially nondecreasing.  The presence of the unspecified factor $s$ allows one to decouple the magnitude of $n$ from the magnitude of $p$, even when $d$ is fixed; as a consequence, this additional monotonicity constraint on $p$ turns out to significantly cut down the possible length of a sequence of numbers of this form in which $\varphi$ is increasing, to the point where the contribution of this case is significantly less than $\frac{x}{\log x}$.

If instead $n=dp$ is of type $A_1$, then we similarly have the approximation
$$ \varphi(n) \approx \frac{\varphi(d)}{d} n,$$
and so the quantity $\frac{\varphi(d)}{d}$ must (heuristically, at least) be locally nondecreasing in order for $\varphi(n)$ to be nondecreasing.  Upon applying the prime number theorem, one can basically conclude that the size of such a sequence is at most
\[ \left(\sup_q \sum_{d: \frac{\varphi(d)}{d}=q} \frac{1}{d} + o(1)\right) \frac{x}{\log x},\]
and one can use Lemma \ref{prelim} to conclude (after performing a more careful accounting of error terms).

The argument used to establish \Cref{thm:main} can also be adapted to treat the analogous problem for the sum of divisors function $\sigma(n)$ as well as the Dedekind totient function $\psi(n)$: see \Cref{divisors-sec}.

The author is supported by NSF grant DMS-1764034.  We thank Wouter van Doorn, Kevin Ford, Paul Pollack and Carl Pomerance for helpful comments, suggestions, and references, and to Shengtong Zhang to providing the key ingredient needed to obtain the analogous results for the sum of divisors function.  We also thank the anonymous referees for several useful comments and references, including \cite{lichtman} and Remark \ref{dedekind-rem}.

\subsection{Notation and basic tools}\label{notation-sec}

We use the usual asymptotic notation $X \ll Y$, $Y \gg X$, or $X = O(Y)$ to denote the claim $|X| \leq CY$ for an absolute constant $C$, $X \asymp Y$ to denote the claim $X \ll Y \ll X$, and $X=o(Y)$ to denote the claim $|X| \leq c(x) Y$ for some quantity $c(x)$ that goes to zero as $x \to \infty$.

The symbol $p$ will be understood to range over primes unless otherwise indicated.

For any $y \geq 2$, we let $\N_{\leq y}$ denote the set of $y$-smooth numbers, that is the natural numbers whose prime factors are all less than or equal to $y$. We similarly let $\N_{<y}$ denotes the natural numbers whose prime factors are all strictly less than $y$.

We record a useful (if somewhat crude) estimate:

\begin{lemma}[Rankin trick]\label{rankin}  If $y \geq 10$, and $a_d$ is a non-negative real number for every $d \in \N_{\leq y}$, then
$$
\sum_{d \in \N_{\leq y}} a_d \ll \log y \sup_{d \in \N_{\leq y}}\left( d^{1-\frac{1}{\log y}} a_d \right).
$$
\end{lemma}

One could obtain sharper bounds here by using standard bounds on smooth numbers here, such as those in \cite{deb}, but the above crude estimate suffices for our application.  In this lemma we allow the quantities on either side to be infinite, though in applications we will only use the lemma when the right-hand side is finite, as the claim is trivial otherwise.

\begin{proof}  By Euler product factorization and Mertens' theorems, we have
\begin{align*}
\sum_{d \in \N_{\leq y}} \frac{1}{d^{1-\frac{1}{\log y}}} &= \prod_{p \leq y} \left( 1 + \frac{p^{1/\log y}}{p} + O\left( \frac{1}{p^{3/2}} \right)\right) \\
&= \exp\left( \sum_{p \leq y} \frac{1}{p} + O\left( \frac{\log p/\log y}{p} \right) + O\left( \frac{1}{p^{3/2}} \right) \right)\\
&= \exp( \log_2 y + O(1) ) \\
&\ll \log y.
\end{align*}
The claim now follows by bounding $a_d \leq \frac{1}{d^{1-\frac{1}{\log y}}}  \sup_{d' \in \N_{\leq y}}\left( (d')^{1-\frac{1}{\log y}} a_{d'}\right)$.
\end{proof}

We also need some well known bounds for primes in intervals:

\begin{lemma}[Primes in intervals]\label{prime}  Let $y \geq 10$, and let $I$ be an interval in $[2,y]$.  Let $|I|$ denote the length of $I$.
Then the number of primes in $I$ is equal to
$$\int_I \frac{\mathrm{d}t}{\log t} + O( y \exp(-c\sqrt{\log y}))$$
for some absolute constant $c>0$.  In particular, we have the crude upper bound of $O(y/\log y)$.
\end{lemma}

\begin{proof} Immediate from the prime number theorem with classical error term.
\end{proof}

As a corollary of the prime number theorem, we obtain a Mertens-type estimate.

\begin{lemma}[Mertens-type bound]\label{mertens} If $2 \leq z \leq y$, then
$$ \sum_{z \leq p \leq y} \frac{1}{p} \ll \frac{\exp(-c \sqrt{\log z}) + \log \frac{y}{z}}{\log z}$$
for an absolute constant $c>0$.
\end{lemma}

\begin{proof} By Mertens' second theorem we have
\begin{align*}
    \sum_{z \leq p \leq y} \frac{1}{p} &= \log\log y - \log\log z + O\left(\frac{1}{\log z}\right) \\
    &\ll \log \left(1 + \frac{\log(y/z)}{\log z}\right) + \frac{1}{\log z}
\end{align*}
which is acceptable if $y > 2z$ (in which case we may discard the $\exp(-c \sqrt{\log z}) $ term).  In the remaining case $z \leq y \leq 2z$,one can bound $\frac{1}{p} \ll \frac{1}{y}$ and apply Lemma \ref{prime}.
\end{proof}

\section{Proof of preliminary inequality}\label{prelim-sec}

We now prove Proposition \ref{prelim}.  It turns out that (somewhat remarkably) one has a reasonably exact formula for the sum of interest, which implies Proposition \ref{prelim} as an immediate corollary:

\begin{lemma}[Exact formula]\label{exact}  Let $q=a/b$ be a positive rational number in lowest terms.  Then the expression $\sum_{d: \frac{\varphi(d)}{d} = q} \frac{1}{d}$ either vanishes, or is equal to $\prod_{p \in {\mathcal P}} \frac{1}{p-1}$, for some finite set of primes ${\mathcal P}$ whose largest element is equal to the largest prime factor of $b$ (with ${\mathcal P}$ empty when $b=1$).  In particular, Proposition \ref{prelim} holds.
\end{lemma}

\begin{proof}  When $b=1$, the only possible value of $q$ that is of the form $\frac{\varphi(d)}{d}$ is $q=1$, which is attained precisely when $d=1$, so the identity holds in this case.

To handle the $b>1$ case we induct on the largest prime factor $p_1$ of $b$, assuming that the claim has already been proven for smaller values of $p_1$ (with the convention that $p_1=1$ when $b=1$).  We may assume that $q$ is of the form $\frac{\varphi(d)}{d}$ for at least one $d$.  For such a $d$, we have
$$ \frac{a}{b} = q = \frac{\varphi(d)}{d} = \prod_{p|d} \frac{p-1}{p}.$$
From this identity we see that the largest prime factor $p_1$ of $b$ must also be the largest prime factor of $d$.  Thus we can write $d = p_1^i d'$ where $i \geq 1$, $d' \in \N_{<p_1}$ and $\frac{\varphi(d')}{d'} = q'$, where $q' = \frac{a'}{b'}$ is equal to $\frac{p_1}{p_1-1} q$ in lowest terms.  By the same computation, the largest prime factor $p'_1$ of $b'$ is equal to that of $d'$, and thus $p'_1 < p_1$.  Conversely, if $d'$ is such that
$\frac{\varphi(d')}{d'} = \frac{a'}{b'}$, then the largest prime factor of $d'$ must equal $p'_1$, so that $d'$ is coprime to $p_1$, and hence any $d$ of the form $d = p_1^i d'$ with $i \geq 1$ is such that
$$ \frac{\varphi(d)}{d} = \frac{p_1-1}{p_1} \frac{\varphi(d')}{d'} = q.$$
Thus we have
$$ \sum_{d: \frac{\varphi(d)}{d} = q} \frac{1}{d} = \sum_{i=1}^\infty \sum_{d': \frac{\varphi(d')}{d'} = q'} \frac{1}{p_1^i d'} = \frac{1}{p_1-1}
\sum_{d': \frac{\varphi(d')}{d'} = q'} \frac{1}{d'}.$$
The claim now follows from the induction hypothesis.
\end{proof}

\begin{remark}\label{eq}  From Lemma \ref{exact} we see that equality in Proposition \ref{prelim} only holds when $q=1$ or $q=\frac{1}{2}$, with the improved inequality $\sum_{d: \frac{\varphi(d)}{d} = q} \frac{1}{d} \leq \frac{1}{2}$ holding in all other cases.  The equality when $q = \frac{1}{2}$ produces an almost viable counterexample to conjectures such as \eqref{mpi}; see Section \ref{dickson-sec} below.
\end{remark}

\begin{remark}  Lemma \ref{exact} is closely related to the well known fact (see e.g., \cite{ford-b}) that if $n,m$ are natural numbers, then $\varphi(n)/n$ and $\varphi(m)/m$ are equal if and only if $n,m$ have exactly the same set of prime factors.
\end{remark}

\section{Main argument}\label{main-sec}

We now prove Theorem \ref{thm:main}.  We may of course take $x$ to be larger than any given absolute constant.  We will need some intermediate scales
\begin{equation}\label{l-scale}
 1 < L < D^{1/\log L} < D < \exp(\sqrt{\log x}) < R < x.
\end{equation}
There is some flexibility in how to select these scales; we will make the choice
\begin{align}
L &\coloneqq \log^{10} x = \exp(10 \log_2 x)\label{L-def} \\
D &\coloneqq \exp(\log_2^3 x) = \left(\exp(\log_2^2 x/10)\right)^{\log L} \label{D-def} \\
R &\coloneqq x^{\frac{1}{3\log_2 x}} = \exp\left( \frac{\log x}{3\log_2 x} \right). \label{R-def}
\end{align}
Note that these scales are widely separated, in the sense that any fixed power of one of the expressions in \eqref{l-scale} is still less than the next expression in \eqref{l-scale}, if $x$ is large enough.  For instance, the savings $\exp(-c\sqrt{\log x})$ arising from Lemma \ref{prime} can easily absorb any polynomial losses in $L$ and $D$.  In a similar spirit, we have $\log \frac{R}{L^2} \asymp \log \frac{R}{L} \asymp \log R$, $\log \frac{x}{DL} \asymp \log \frac{x}{D} \asymp \log x$, $\log(DL) \asymp \log D$, and so forth.

We observe the triangle inequality\footnote{This is of course a rather crude inequality, as it concedes the ability to exploit the monotonicity of $\varphi$ when comparing an element of $A$ with an element of $B$.  To make significant improvements to \eqref{M-pi}, one should probably avoid using this inequality, and instead take advantage of the ``global'' monotonicity of $\varphi$ on the entire sequence under consideration.}
\begin{equation}\label{triangle}
M(A \cup B) \leq M(A) + M(B)
\end{equation}
as well as the trivial bound
\begin{equation}\label{trivial}
M(A) \leq \# A.
\end{equation}
We shall cover the set $[x]$ by three (slightly overlapping) sets
\begin{equation}\label{decomp}
[x] = A_1 \cup A_2 \cup E
\end{equation}
which by \eqref{triangle}, \eqref{trivial} implies that
\begin{equation}\label{triangle-split}
 M(x) \leq M(A_1) + M(A_2) + \# E.
\end{equation}
We will control each of the three terms on the right-hand side separately.

The exceptional set $E$ is defined to be the set of all $n \in [x]$ having an unusual prime factorization in at least one of the following senses:
\begin{itemize}
    \item[(i)] (Smallness) $n \leq \frac{x}{L}$.
    \item[(ii)] (Smoothness) $n \in \N_{\leq R}$.
    \item[(iii)] (Large square factor) $n$ is divisible by $d^2$ for some $d>L$.
    \item[(iv)] (Large smooth factor) $n$ is divisible by some $d \in \N_{\leq L}$ with $d > D$.
    \item[(v)]  (Smooth times two nearby primes) $n = d p_2 p_1$ where $d \in \N_{\leq L}$ and $p_1,p_2$ are primes with $\frac{R}{L} \leq p_2 \leq p_1 \leq p_2 L$.
    \item[(vi)] (Integer times three nearby primes) $n = d p_3 p_2 p_1$ where $p_1,p_2,p_3$ are primes with $\frac{R}{L^2} \leq p_3 \leq p_2 \leq p_1 \leq p_3 L^2$.
\end{itemize}
The primary set $A_1$ is defined to be the set of all $n \in [x]$ of the form
$$ n = dp$$
where $d \in \N_{\leq L}$ with $d \leq D$, $p$ is a prime, and $n > \frac{x}{L}$.  The secondary set $A_2$ is defined to be the set of all $n \in [x]$ of the form
$$ n = dps$$
where $p$ is a prime greater than $L$, $d \in \N_{<p}$, and $s$ is an almost prime (either a prime or product of two primes) whose prime factors all exceed $pL$.

\begin{lemma}[Decomposition]\label{lem-decomp} If $n \in [x]$, then $n$ lies in at least one of $A_1, A_2, E$.   In other words, \eqref{decomp} (and hence \eqref{triangle-split}) holds.
\end{lemma}

\begin{proof}  We may assume that $n$ avoids all the cases (i)--(vi) of $E$.  Let $p_1 \geq p_2 \geq p_3 \geq p_4$ be the four largest prime factors of $n$ (with the convention that $p_i=1$ if $n$ has fewer than $i$ prime factors).  Since $n$ avoids the cases (i), (ii) of $E$ we have $n > \frac{x}{L}$ and $p_1 > R$.  We divide into three cases.
\begin{itemize}
    \item If $p_2 > \frac{p_1}{L}$, then we must have $p_3 \leq \frac{p_2}{L}$ (since $n$ avoids case (vi) of $E$), and $p_3 > L$ (since $n$ avoids case (v) of $E$).  We then have $p_4 \neq p_3$ (since $n$ avoids case (iii) of $E$).  One then checks that one is in the set $A_2$ (with $p = p_3$, $s = p_2 p_1$, and $d = \frac{n}{p_1 p_2 p_3}$).
    \item If $L < p_2 \leq \frac{p_1}{L}$, then we must have $p_3 \neq p_2$ (since $n$ avoids case (iii) of $E$).  One then checks that one is in the set $A_2$ (with $p=p_2$, $s=p_1$, and $d = \frac{n}{p_1 p_2}$).
    \item Finally, if $p_2 \leq L$, then the quantity $d \coloneqq \frac{n}{p_1}$ is at most $D$ (since $n$ avoids case (iv) of $E$), and then one checks that one is in the set $A_1$ (with $p=p_1$).
\end{itemize}
\end{proof}

Now we verify that the exceptional set $E$ is small:

\begin{proposition}[$E$ small]\label{prop-e}  We have $\# E \ll \frac{x \log^5_2 x}{\log^2 x}$.
\end{proposition}

\begin{proof}  It suffices to show that the set of $n \in [x]$ obeying each of the six cases (i)-(vi) of $E$ has size $O( \frac{x \log^5_2 x}{\log^2 x})$.

The set of $n$ obeying (i) clearly has size $O(x/L)$, which is acceptable (with plenty of room to spare).  By Lemma \ref{rankin}, the set of $n$ obeying (ii) has size at most\footnote{One can in fact obtain much stronger bounds here using \cite{deb}.}
\begin{align*}
    \sum_{n \in \N_{\leq R}: n \leq x} 1 \ll x^{1-\frac{1}{\log R}} \log R
\end{align*}
which is also acceptable (with room to spare) from  \eqref{R-def}.  The set of $n$ obeying (iii) has size at most
$$ \sum_{L < d \leq x^{1/2}} \frac{x}{d^2} \ll \frac{x}{L}$$
which is again acceptable (with plenty of room to spare).

The set of $n$ obeying (iv) has size at most
$$\sum_{d \in \N_{\leq L}: d > D} \frac{x}{d}.$$
Applying Lemma \ref{rankin}, we can bound this by
$$ \ll x \frac{\log L}{D^{1/\log L}},$$
which is acceptable (with substantial room to spare) from \eqref{L-def}, \eqref{D-def}.

Now suppose that $n$ obeys (v), thus
$$ n = d p_2 p_1$$
where $d \in \N_{\leq L}$, and $p_1, p_2$ are primes with
$$ \frac{R}{L} \leq p_2 \leq p_1 \leq \min \left( \frac{x}{d p_2}, p_2 L \right).$$
In particular $p_2 \leq \sqrt{\frac{x}{d}}$.  We may also assume that $d \leq D$, since otherwise we are in case (iv).

By the prime number theorem (Lemma \ref{prime}), we see that for any choice of $d, p_2$, the number of possible values of $p_1$ is at most
$$ \ll \frac{1}{\log \frac{R}{L}} \min\left( \frac{x}{d p_2}, p_2 L\right).$$
We conclude that this contribution to the size is at most
$$ \ll \frac{1}{\log \frac{R}{L}} \sum_{d \in \N_{\leq L}: d \leq D} \sum_{p_2 \leq \sqrt{\frac{x}{d}}}
\min\left( \frac{x}{d p_2}, p_2 L\right).$$
From the prime number theorem (Lemma \ref{prime}) we have
\begin{align*}
 \sum_{p_2 < \sqrt{\frac{x}{dL}}} p_2 L &\leq L \sqrt{\frac{x}{dL}} \pi\left( \sqrt{\frac{x}{dL}} \right) \\
&\ll L \frac{x}{dL} \frac{1}{\log \frac{x}{dL}} \\
&\leq \frac{1}{\log \frac{x}{DL}} \frac{x}{d}
\end{align*}
while from Lemma \ref{mertens} we have
$$ \sum_{\sqrt{\frac{x}{dL}} \leq p_2 < \sqrt{\frac{x}{d}}} \frac{x}{dp_2} \ll \frac{\log L}{\log \frac{x}{DL}} \frac{x}{d}$$
and so by we can bound the previous contribution to $\# A_3$ by
$$ \ll  \frac{x \log L}{\log \frac{R}{L} \log \frac{x}{DL}} \sum_{d \in \N_{\leq L}} \frac{1}{d}$$
which by\footnote{Alternatively, we can use the Euler product $\sum_{d \in \N_{\leq L}} \frac{1}{d} = \prod_{p\leq L} (1-\frac{1}{p})^{-1}$ followed by Mertens' theorem.} Lemma \ref{rankin} is bounded by
$$ \ll  \frac{x \log^2 L}{\log \frac{R}{L} \log \frac{x}{DL}} .$$
This is acceptable thanks to \eqref{L-def}, \eqref{D-def}, \eqref{R-def}.

Finally, if $n$ obeys (vi), then we may factor
$$ n = d p_3 p_2 p_1$$
where $p_1,p_2,p_3$ are primes in $[x]$ with
$$ \frac{R}{L^2} \leq p_3 \leq p_2 \leq p_1 \leq p_3 L^2.$$
If we fix $p_1,p_2,p_3$, there are at most $\frac{x}{p_1 p_2 p_3}$ choices of $d$, thus this contribution to the size is at most
$$ \sum_{\frac{R}{L^2} \leq p_3 \leq x} \sum_{p_2,p_1: p_3 \leq p_2 \leq p_1 \leq p_3 L^2} \frac{x}{p_1 p_2 p_3}.$$
By two applications of Lemma \ref{mertens}, this is
$$ \ll \frac{x \log^2(L^2)}{\log^2(R/L^2)} \sum_{p_3 \leq x} \frac{1}{p_3},$$
which is acceptable thanks to Mertens' theorem and \eqref{R-def}, \eqref{L-def}.
\end{proof}

We remark that one could obtain much sharper asymptotics on the size of $E$ if desired using more sophisticated results from the anatomy of integers, such as \cite{ford-a}.

Now we control $A_2$.

\begin{proposition}[$A_2$ negligible]\label{a5}  We have $M(A_2) \ll \frac{x}{\log^2 x}$.
\end{proposition}

\begin{proof}  Let $A'$ be a subset of $A_2$ on which $\varphi$ is nondecreasing.  Our task is to show that
$$ \# A' \ll \frac{x}{\log^2 x}.$$
This will be the first time we will use the monotonicity of $\varphi$ in our arguments; on the other hand, there is a lot of room here, and we can afford to lose several powers of $\log x$ in what follows. If $n \in A'$, then we may factor
$$ n = d p s$$
where $p$ is a prime in the range
$$ L < p \leq x^{1/2},$$
$d$ is an element of $\N_{<p}$ with $d \leq \frac{x}{p^2 L}$, and $s$ is an almost prime, with all prime factors larger than $p L$.  In particular
$$ \varphi(s) = \left(1 - O\left(\frac{1}{p L}\right)\right) s$$
and thus by multiplicativity
\begin{align}
\varphi(n) &= \varphi(d) \varphi(p) \varphi(s) \nonumber \\
&= \frac{\varphi(d)}{d} \left(1 - \frac{1 + O(\frac{1}{L})}{p}\right) n.\label{varphin}
\end{align}
We exploit this relation as follows.  First we perform dyadic decomposition to estimate
$$
\# A' \leq \sum_{m: L < 2^m \ll x^{1/2}} \sum_{d \in \N_{<2^m}: d \ll \frac{x}{2^{2m}L} } a(m,d)$$
where
$$ a(m,d) \coloneqq \sum_{2^{m-1} < p \leq 2^m}  \sum_{s \leq \frac{x}{dp}: dp s \in A'; \eqref{varphin}} 1$$
and the inner sum is restricted to those $s$ for which $n = dp s$ obeys the estimate \eqref{varphin}.  We will establish the estimate
\begin{equation}\label{md}
a(m,d) \ll \frac{x}{d \log^4 x}
\end{equation}
for each $m$ with $L < 2^m \ll x^{1/2}$ and all $d \in \N_{<2^m}$; the claim will then follow by summing \eqref{md} over all $d,m$ and using crude estimates (or Lemma \ref{rankin}).

It remains to establish \eqref{md}.  We fix $m,d$.
We cover the interval $(1,x]$ by $O(2^m\log^6 x)$ consecutive disjoint intervals $I_i$ of the form $(N, (1 + \frac{1}{2^m \log^5 x}) N]$ with $1 \leq N \ll x$.  We similarly cover the interval $(2^{m-1}, 2^m]$ by $O(\log^4 x)$ consecutive disjoint intervals $J_k$ of the form $(M, (1 + \frac{1}{\log^4 x})M]$ with $M \asymp 2^m$.  For each $I_i$ and $J_k$, let $I_{i,k}$ be the smallest interval containing all the $n \in I_i$ of the form $n = dp s$, where $p, s$ contribute to \eqref{md} and $p \in J_k$.  (We allow $I_{i,k}$ to be empty, or a single point.)  Then we can bound
\begin{equation}\label{md-2}
a(m,d) \leq
\sum_i \sum_k \sum_{p \in J_k}
\sum_{s: dps \in I_{i,k}} 1.
\end{equation}
Since $I_{i,k} \subset I_i$ and the $I_i$ are disjoint, then $I_{i,k}$ and $I_{i',k'}$ are disjoint for $i \neq i'$.
Now we exploit the monotonicity of $\varphi$ on $A'$ via a key observation: if $k' \geq k+2$, then $I_{i,k'}$ lies strictly to the right of $I_{i,k}$ (with the convention that this claim is vacuously true if at least one of $I_{i,k}$, $I_{i,k'}$ is empty).  Indeed, suppose that $n = d p s \in I_{i,k}$ and $n' = d p' s' \in I_{i,k'}$ contribute to \eqref{md} with $p \in J_k$ and $p' \in J_{k'}$.  From \eqref{varphin} we have
$$
\varphi(n) = \frac{\varphi(d)}{d} \left(1 - \frac{1 + O(\frac{1}{L})}{p}\right) n$$
and
$$
\varphi(n') = \frac{\varphi(d)}{d} \left(1 - \frac{1 + O(\frac{1}{L})}{p'}\right) n'$$
On the other hand, since $n, n' \in I_i$ and $p' \asymp 2^m$, we have
$$ n' = \left(1 + O\left(\frac{1}{p' \log^5 x}\right)\right) n$$
and thus (by \eqref{L-def})
$$
\varphi(n') = \frac{\varphi(d)}{d} \left(1 - \frac{1 + O(\frac{1}{\log^5 x})}{p'}\right) n.$$
Since $p \in J_k$, $p' \in J_{k'}$ and $k' \geq k+2$ we have
$$ p' \geq \left(1 + \frac{1}{\log^4 x}\right) p.$$
Putting these estimates together (and recalling that $x$ is assumed to be large), we conclude that
$$ \varphi(n') > \varphi(n)$$
and hence by monotonicity of $\varphi$ on $A'$ we have $n' > n$.  This gives the claim.

As a corollary of this monotonicity, we see that the intervals $I_{i,k}$ have bounded overlap in $[1,x]$ as $i, k$ vary (in fact any $n$ can belong to at most two of the $I_{i,k}$).  In particular,
\begin{equation}\label{ikx}
\sum_i \sum_k |I_{i,k}| \ll x.
\end{equation}

We have the\footnote{One could save an additional factor of $\log x$ or so here by using a Brun--Titchmarsh inequality for almost primes, but this turns out not to significantly improve the bounds in our main result.} trivial bound
$$ \sum_{s: dp s \in I_{i,k}} 1 \ll \frac{|I_{i,k}|}{dp} + 1$$
and hence by \eqref{md-2} we may bound
$$
a(m,d) \ll \sum_i \sum_k \sum_{p \in J_k} \left( \frac{|I_{i,k}|}{dp} + 1 \right).$$
Since there are only $O(2^m \log^6 x)$ values of $i$, and the $J_k$ cover an interval $\{ y: y \asymp 2^m\}$ with bounded overlap, we can crudely estimate
$$ \sum_i \sum_k \sum_{p \in J_k} 1 \ll 2^{2m} \log^6 x$$
which since $d \ll \frac{x}{2^{2m} L}$ implies that
$$ \sum_i \sum_k \sum_{p \in J_k} 1 \ll \frac{x}{d} \frac{\log^6 x}{L}.$$
Next, by discarding the requirement that $p$ is prime, we have the trivial bound
$$ \sum_{p \in J_k} 1 \ll \frac{2^m}{\log^4 x} + 1;$$
since $p \asymp 2^m > L$ for $p \in J_k$, we conclude
$$ \sum_{p \in J_k} \frac{1}{dp} \ll \frac{1}{d} \left( \frac{1}{\log^4 x} + \frac{1}{L} \right).$$
By \eqref{ikx}, we conclude that
$$ \sum_i \sum_k \sum_{p \in J_k} \frac{|I_{i,k}|}{dp} \ll \frac{x}{d} \left( \frac{1}{\log^4 x} + \frac{1}{L} \right).$$
Combining these estimates, we obtain
$$ a(m,d) \ll \frac{x}{d} \left( \frac{1}{\log^4 x} + \frac{\log^6 x}{L} \right)$$
and \eqref{md} follows from \eqref{L-def}.
\end{proof}

Finally, we treat the primary set $A_1$.

\begin{proposition}[$A_1$ contribution]\label{a7}  We have $M(A_1) \leq \left(1 + O\left( \frac{\log_2 x}{\log x} \right) \right) \frac{x}{\log x}$.
\end{proposition}

\begin{proof}  Let $A'$ be a subset of $A_1$ on which $\varphi$ is nondecreasing.  Our task is to show that
\begin{equation}\label{a7-task}
\# A' \leq \left(1 + O\left( \frac{\log_2 x}{\log x} \right) \right)  \frac{x}{\log x}.
\end{equation}

If $n \in A'$, then we may factor
$$ n = dp$$
where $d \in \N_{\leq L}$ with
\begin{equation}\label{d-upper}
d \leq D
\end{equation}
and $p$ is a prime which then lies in the range
\begin{equation}\label{p1-lower}
 \frac{x}{DL} \leq p \leq \frac{x}{d}.
\end{equation}
In particular (from \eqref{D-def}) we certainly have
\begin{equation}\label{p1-lower-2}
p \geq D^3.
\end{equation}
Thus $p$ is coprime to $d$, and hence
\begin{align}
\varphi(n) &= \varphi(d) \varphi(p)\nonumber \\
&= \frac{\varphi(d)}{d} \left( 1 - \frac{1}{p} \right) n.\label{vnd}
\end{align}
If we define the set
$$ {\mathcal Q} \coloneqq \left\{ \frac{\varphi(d)}{d}: d \in \N_{\leq L}; d \leq D \right\}$$
then ${\mathcal Q}$ is a finite set of rationals in $(0,1]$, of some cardinality $K = \# {\mathcal Q}$ with
\begin{equation}\label{KB}
K \leq D.
\end{equation}
We arrange the elements of ${\mathcal Q}$ in increasing order as
$$ 0 < q_1 < q_2 < \dots < q_K \leq 1.$$
For each $k=1,\dots,K$, we let ${\mathcal D}_k$ denote the set of all $d \in \N_{\leq L}$ with $d \leq D$ and
\begin{equation}\label{varphi-ratio}
\frac{\varphi(d)}{d} = q_k.
\end{equation}
We also partition $[\frac{x}{L}, x]$ into $O(D^3 \log L)$ disjoint intervals $I_i$ of the form $(N, (1+O(\frac{1}{D^3})) N]$.  For $k=1,\dots,K$, and any $i$, let $I_{i,k}$ denote the smallest interval containing all $n=dp$ in $A' \cap I_i$ with $d \in {\mathcal D}_k$ and $p$ a prime obeying \eqref{p1-lower}.  Then we have
\begin{equation}\label{a7p}
\# A' \leq \sum_i \sum_{k=1}^K \sum_{d \in {\mathcal D}_k} \sum_{p: d p \in I_{i,k}} 1.
\end{equation}
The key observation is that the intervals $I_{i,k}$ are all disjoint.  To see this, first observe that since $I_{i,k} \subset I_i$ and the $I_i$ are disjoint, we clearly have $I_{i,k} \cap I_{i',k'} = \emptyset$ if $i \neq i'$.  Next, we claim that if $1 \leq k < k' \leq K$ then $I_{i,k'}$ lies strictly to the right of $I_{i,k}$ (so in particular the two sets are disjoint).  Indeed, if $n = d p \in I_{i,k}$ and $n' = d' p' \in I_{i,k'}$ then from \eqref{vnd}, \eqref{p1-lower-2}, \eqref{varphi-ratio} one has
$$ \varphi(n) = \frac{\varphi(d)}{d} \left(1 - \frac{1}{p}\right)n = q_k \left(1 - O\left(\frac{1}{D^3}\right)\right)n.$$
Similarly we have
$$ \varphi(n') = q_{k'}\left(1 - O\left(\frac{1}{D^3}\right)\right) n'.$$
Also, as $n,n'$ both lie in $I'$, we have
$$ n' = \left(1 + O\left(\frac{1}{D^3}\right)\right) n,$$
while since $k' > k$, we have $q_{k'} > q_k$.  From \eqref{d-upper}, the rational numbers $q_{k'}, q_k$ have denominator at most $D$, hence
$$ q_{k'} - q_k \geq \frac{1}{D^2}$$
and thus (since $q_k \leq 1$)
$$ q_{k'} \geq \left(1 + \frac{1}{D^2}\right) q_k.$$
Combining all these statements (and recalling that $x$ and hence $D$ is large), we conclude that
$$ \varphi(n') > \varphi(n)$$
and hence $n' > n$ by monotonicity, giving the claim.

From the prime number theorem (Lemma \ref{prime}) and \eqref{p1-lower} we have
\begin{align*}
\sum_{p: d p \in I_{i,k}} 1 &\leq \int_{\frac{1}{d} I_{i,k}} \frac{\mathrm{d}t}{\log t} + O\left( \frac{x}{d} \exp\left(-c \sqrt{\log \frac{x}{d}}\right) \right)\\
&\leq \left(1 + O\left( \frac{\log(DL)}{\log x} \right) \right)  \frac{|I_{i,k}|}{d \log x} + O\left( \frac{x}{d} \exp\left(-c \sqrt{\log \frac{x}{D}}\right) \right)
\end{align*}
for an absolute constant $c>0$. From Lemma \ref{prelim} we have the crucial inequality
$$ \sum_{d \in {\mathcal D}_k} \frac{1}{d} \leq 1.$$
Inserting these bounds back into \eqref{a7p}, we conclude that
$$\# A' \leq \left(1 + O\left( \frac{\log(DL)}{\log x} \right) \right) \sum_i \sum_{k=1}^K \frac{|I_{i,k}|}{\log x} + O\left( \sum_i \sum_{k=1}^K x \exp\left(-c \sqrt{\log \frac{x}{D}}\right) \right).$$
From disjointness of the $I_{i,k}$ we have
$$ \sum_i \sum_{k=1}^K |I_{i,k}| \leq x.$$
Since there are $O(D^3 \log L)$ values of $i$, we conclude
$$\# A' \leq \left(1 + O\left( \frac{\log(DL)}{\log x} \right) \right) \frac{x}{\log x} + O\left( x K D^3 (\log L) \exp\left(-c \sqrt{\log \frac{x}{D}}\right) \right)$$
and \eqref{a7-task} follows from \eqref{KB}, \eqref{L-def}, \eqref{D-def}.
\end{proof}

Inserting Propositions \ref{prop-e}-\ref{a7} into \eqref{triangle-split}, we obtain Theorem \ref{thm:main} as claimed.

\section{Obstacles to improvement, variants, and open questions}\label{obstruct}

It seems likely that small improvements to Theorem \ref{thm:main}, such as the lowering of the exponent $5$ in the power of $\log_2 x$, could be achieved with minor modifications of the method.  However, we now present two obstacles that seem to prevent one from unconditionally establishing significantly stronger results, such as \eqref{mpi}, and pose some open problems that might still be achievable despite these obstacles.

\subsection{Legendre conjecture obstruction}

\emph{Legendre's conjecture} asserts that for every $n \geq 2$, there is a prime between $(n-1)^2$ and $n^2$.  This conjecture remains open even for sufficiently large $n$, and even if one assumes the Riemann hypothesis (RH); it is roughly equivalent to showing that the gap $p_{n+1}-p_n$ between consecutive primes is significantly smaller than $p_n^{1/2}$, whereas the best known bound is $O(p_n^{1/2} \log p_n)$ on RH \cite{cramer}, or $O(p_n^{0.525})$ unconditionally \cite{bhp}.  On the other hand, this conjecture is implied by the Cram\'er conjecture $p_{n+1}-p_n \ll \log^2 p_n$ \cite{cramer}; it has also been verified for $n \leq 2 \times 10^9$ \cite{silva}.

Legendre's conjecture is also open if one restricts $n$ to be prime.  This special case of the conjecture will need to be resolved if one wishes to establish \eqref{mpi}, due to the following implication:

\begin{proposition}\label{leg}  Suppose that Legendre's conjecture is false for infinitely many primes; that is to say, there exist infinitely many primes $p$ such that there is no prime between $(p-1)^2$ and $p^2$.  Then $M(x) - \pi(x)$ goes to infinity as $x \to \infty$.
\end{proposition}

Stated contrapositively: if \eqref{mpi} holds, then Legendre's conjecture is true for all sufficiently large primes.

\begin{proof}  If $p$ is a prime counterexample to Legendre's conjecture, then for primes $p'$ larger than $p^2$ we have
$$ \varphi(p') = p'-1 \geq p^2 > \varphi(p^2)$$
and for primes $p'$ less than $p^2$, we have $p' < (p-1)^2$ and hence
$$ \varphi(p') < p' < (p-1)^2 < p^2-p = \varphi(p^2).$$
Thus one can insert $p^2$ into the sequence of primes without disrupting the monotonicity of $\varphi$.  If there are infinitely many counterexamples to Legendre's conjecture at the primes, then one can then insert an infinite number of additional elements into the primes to create a larger sequence with $\varphi$ nondecreasing, giving the claim.
\end{proof}

\begin{remark} One can replace Legendre's conjecture here with the very slightly stronger conjecture of Oppermann \cite{oppermann} that there is always a prime between $n(n-1)$ and $n^2$ for $n \geq 3$.  (Oppermann also made the analogous conjecture concerning $n^2$ and $n(n+1)$, but that is of less relevance here.)
\end{remark}

\begin{remark} With a little more computation, one can show that the conjecture \eqref{pi-64} implies that Legendre's conjecture holds at \emph{all} primes (not just sufficiently large ones).
\end{remark}

\begin{remark} In \cite{selberg}, the estimate
$$ \sum_{n \leq x} \frac{(p_{n+1}-p_n)^2}{p_n} \ll \log^3 x$$
is proven for large $x$ assuming the Riemann hypothesis, where $p_n$ denotes the $n^{\mathrm{th}}$ prime.  One consequence of this is that the counterexamples to Legendre's conjecture are quite rare on RH, since each such counterexample contributes $\gg 1$ to this sum.  Thus this obstruction only worsens the gap between $M(x)$ and $\pi(x)$ by at most $O(\log^3 x)$, on RH.
\end{remark}

\subsection{Dickson--Hardy--Littlewood conjecture obstruction}\label{dickson-sec}

The Dickson--Hardy--Littlewood conjecture \cite{dickson}, \cite{hardy} asserts, among other things, that if $a,b$ are coprime integers with $a>0$, that the number of primes $p \in [x]$ with $ap+b$ also prime is equal to
$$ ({\mathfrak S}_{a,b} + o(1)) \frac{x}{\log^2 x}$$
as $x \to \infty$ for some positive quantity ${\mathfrak S}_{a,b}$ (known as the \emph{singular series}) which has an explicit form which we will not give here.  Among other things, this conjecture (together with standard bounds for averages of singular series) would imply for any fixed natural number $k$ that the number of primes $p \in [x]$ such that the quantity $\lceil \frac{p}{2^k} \rceil$ (the first integer greater than or equal to $p/2^k$) is also prime is $\asymp \frac{x}{\log^2 x}$, if $x$ is sufficiently large depending on $k$.  Such results are out of reach of current technology; note for instance that the Dickson--Hardy--Littlewood conjecture implies the notoriously open twin prime conjecture.

The following result shows that a severe breakdown of the Dickson--Hardy--Littlewood conjecture would imply that the bounds in Theorem \ref{thm:main} cannot be improved (except perhaps for removing the factors of $\log_2 x$).

\begin{proposition}  Suppose that for all sufficiently large $x$, and all natural numbers $k$ with $2^k \leq 2\log x$, that the number of primes $p \in [x]$ with $\lceil \frac{p}{2^k} \rceil$ also prime is bounded by $O(\frac{x}{\log^2 x \log^3_2 x})$ (in contradiction with what the Dickson--Hardy--Littlewood conjecture would predict).  Then one has
$$ M(x) - \pi(x) \gg \frac{x}{\log^2 x}$$
for sufficiently large $x$.
\end{proposition}

This proposition suggests that significant progress on the Dickson--Hardy--Littlewood conjecture would need to be made in order to improve the error term in Theorem \ref{thm:main} to be better than $O(\frac{1}{\log x})$; in particular, this would need to be accomplished in order to have any hope of establishing \eqref{mpi}.  Note that standard probabilistic heuristics (cf. \cite{gallagher}) predict that for $2^k \leq 2 \log x$, a positive fraction of the intervals $[n, n+2^k]$ in $[x]$ will be devoid of primes, and so the set $\{ \lceil \frac{p}{2^k} \rceil: p \in [x] \}$ should avoid a positive density subset of $[x/2^k]$.  Thus, in the absence of the Dickson--Hardy--Littlewood conjecture  it is \emph{a priori} conceivable that the primes in $[x/2^k]$ mostly avoid such a set; we do not know how to rule out such a scenario with current technology.
We remark that the Maynard sieve \cite[Theorem 3.1]{maynard} (when combined with the calculations in \cite{polymath}; see also \cite[Lemma 2]{ford-b}) applied to the linear functions $n \mapsto 2^j n + 1$ for $j=1,\dots,50$ implies that there exists $1 \leq k \leq 49$ such that the number of primes $p$ in $[x]$ with $\lceil \frac{p}{2^k} \rceil$ also prime is $\gg \frac{x}{\log^{50} x}$ for infinitely many $x$, but this bound is not strong enough to contradict the hypothesis of the above proposition.

\begin{proof}  The construction here is motivated by the observation (see Remark \ref{eq}) that Proposition \ref{prelim} holds with equality not only at $q=1$, but also at $q=\frac{1}{2}$, as well as the construction discussed after Proposition \ref{prelim} in the introduction.  We thank Kevin Ford for pointing out an improvement to the argument that gave superior quantitative bounds.

Let $x$ be sufficiently large, and let $k_0$ be the unique natural number with
$$\log x < 2^{k_0} \leq 2 \log x.$$
Let $A$ be the set of all numbers $n \in [x]$ of the form $n = 2^k p$ where $1 \leq k \leq k_0$ and $p$ is an odd prime.  Observe from the prime number theorem (Lemma \ref{prime}) and the logarithmic integral asymptotic
$$ \int_2^x \frac{\mathrm{d}t}{\log t} = \frac{x}{\log x} + \frac{x}{\log^2 x} + O\left( \frac{x}{\log^3 x}\right)$$
followed by Taylor expansion that
\begin{align*}
\# A &= \sum_{k=1}^{k_0} (\pi(x/2^k) - 1) \\
&= \sum_{k=1}^{k_0} \left( \frac{x/2^k}{\log(x/2^k)} + \frac{x/2^k}{\log^2(x/2^k)} + O\left( \frac{x/2^k}{\log^3(x/2^k)}  \right) \right) \\
&= \sum_{k=1}^{k_0} \left( \frac{x}{2^k(\log x - k \log 2)} + \frac{x}{2^k(\log x - k \log 2)^2} + O\left( \frac{x}{2^k \log^3 x}  \right)  \right)\\
&= \sum_{k=1}^{k_0} \left( \frac{x}{2^k \log x} + \frac{(1 + k \log 2) x}{2^k \log^2 x} + O\left( \frac{k^2 x}{2^k \log^3 x}  \right) \right).
\end{align*}
We have $k_0 = O(\log_2 x)$,
$$ \sum_{k=1}^{k_0} \frac{1}{2^k} = 1 - \frac{1}{2^{k_0}} \geq 1 - \frac{1}{\log x}$$
and
$$ \sum_{k=1}^{k_0} \frac{k}{2^k} = 2 - O\left( \frac{k_0}{2^{k_0}} \right) = 2 - O\left( \frac{\log_2 x}{\log x} \right)$$
and thus
\begin{equation}\label{ap}
\# A \geq \frac{x}{\log x} + \frac{x \log 4}{\log^2 x} - O\left( \frac{x \log_2 x}{\log^3 x} \right).
\end{equation}
Note for comparison that
$$ \pi(x) = \frac{x}{\log x} + \frac{x}{\log^2 x} + O\left( \frac{x}{\log^3 x} \right).$$
Since $\log 4 > 1$, we see that $A$ is slightly denser than the primes.

Unfortunately, $\varphi$ is not quite nondecreasing on $A$; however, the exceptions to monotonicity are quite rare and can be described explicitly as follows.  Suppose that $2^k p$, $2^{k'} p'$ are elements of $A$ with $2^k p < 2^{k'} p'$ and $\varphi(2^k p) > \varphi(2^{k'} p')$.  Evaluating the totient functions, we obtain
$$ 2^{k-1} (p-1) > 2^{k'-1} (p'-1)$$
which we can rearrange as
$$ 2^{k'} p' - 2^k p < 2^{k'} - 2^k.$$
Since the left-hand side is positive, we must then have $k' > k$.  Dividing by $2^{k'}$, we conclude in particular that
$$ \frac{p}{2^{k'-k}} < p' < \frac{p}{2^{k'-k}}+1$$
and thus
$$ p' = \left\lceil \frac{p}{2^{k'-k}} \right\rceil.$$
By hypothesis, this scenario can only occur for $O(\frac{x}{\log^2 x \log^3_2 x})$ primes $p$ for each choice of $k,k'$.  Since there are at most $O(\log_2^2 x)$ choices for $k,k'$, we thus see that by deleting at most $O(\frac{x}{\log^2 x \log_2 x})$ elements from $A$, one can obtain a slightly smaller set $A'$ that still obeys the asymptotic \eqref{ap}, but on which $\varphi$ is now nondecreasing.  Since
$$ \# A' - \pi(x) \gg \frac{x}{\log^2 x}$$
for sufficiently large $x$, we obtain the claim.
\end{proof}

\subsection{A conditional result?}

Despite these two obstructions, it is possible that a more refined version of the analysis in this paper can be used to show that a significant improvement to the bounds in Theorem \ref{thm:main} can be achieved \emph{unless} one of the above two scenarios occurs (breakdown of Legendre-type conjectures, or breakdown of Dickson--Hardy--Littlewood type conjectures).  Indeed, the fact (see Remark \ref{eq}) that Lemma \ref{prelim} gains a factor of two when $q \neq 1, \frac{1}{2}$ implies (by a routine modification of the arguments in this paper) that, if one sets $A$ to be the elements of $[x]$ that are not a prime, or a power of two times a prime, then $M(A) \leq (\frac{1}{2} + O( \frac{\log_2^5 x}{\log x} )) \pi(x)$.  Thus largest subsets of $[x]$ on which $\varphi$ are increasing must have almost half or more of their elements consist of primes, or powers of two times a prime.  With a suitable form of the Dickson--Hardy--Littlewood conjectures, one might be able to show that the sets which contain a large number of powers of two times a prime lead to a set with fewer than $\pi(x)$ elements, and so one would be left with studying sets that largely consist of primes.  Here, one could take advantage of the phenomenon (already observed in \cite[\S 9]{pomerance}) that for any prime $p$, the first composite number $n > p$ with $\varphi(n) \geq \varphi(p)$ must in fact exceed $p + \sqrt{p} - 1$, which suggests that (in the absence of extremely large prime gaps, as in Proposition \ref{leg}), it could rule out competitors to the set of primes in which one or more primes in the ``interior'' of the set are replaced with composite numbers.  In view of known results on the anatomy of integers in short intervals, one might hope to obtain a bound such as $M(x) = \pi(x) + O( x^\theta )$ for some $0 < \theta < 1$ assuming additional conjectures, such as a sufficiently quantitative version of the Dickson--Hardy--Littlewood conjecture.  In a closely related spirit, one could also seek to obtain a ``short interval'' version of the main theorem of the form
$$ M( [x + x^\theta] \backslash [x] ) \leq (1+o(1)) (\pi(x+x^\theta) - \pi(x))$$
for some $0 < \theta < 1$.

\subsection{Analogues for the sum of divisors function}\label{divisors-sec}

It turns out (as already conjectured in \cite{erdos}) that the Euler totient function $\varphi(n)$ may be replaced\footnote{See \url{https://oeis.org/A365398} for the analogue of $M(x)$ for the sum of divisors function.} by the sum of divisors function $\sigma(n) \coloneqq  \sum_{d|n} d$ throughout the above arguments.  The main difficulty is to obtain an analogue
\begin{equation}\label{pol}
\sum_{\frac{\sigma(d)}{d} = q} \frac{1}{d} \leq 1
\end{equation}
of Proposition \ref{prelim} for the sum of divisor function, since there is no analogue of Lemma \ref{exact}.  For instance, $\sum_{\frac{\sigma(d)}{d} = 2} \frac{1}{d}$ is the sum of reciprocals of perfect numbers, for which there are notorious open questions (for instance, it is still open whether there are infinitely many perfect numbers), though in this case $q=2$ the sum in \eqref{pol} is known to equal $0.204520\dots$ \cite[Theorem 7]{bayless}. On the other hand, the weaker bound
$$\sum_{\frac{\sigma(d)}{d} = q} \frac{1}{d} \ll 1$$
was observed\footnote{We thank Paul Pollack for these references.} in \cite[p. 63]{erdos-dist} to follow from the estimate claimed in \cite[Theorem 3]{erdos-remarks} (see also \cite[p. 212--213]{elliott} for an alternate proof).  In \cite[p. 63]{erdos-dist} it was also claimed that \eqref{pol} could be shown ``with more trouble'', but without a full proof provided.  Such problems are also related to those of counting solutions to $\sigma(n)=\sigma(n+k)$ for fixed $k$; see \cite{ford-b}.  Since the initial release of this preprint, a full proof of \eqref{pol} using the methods in \cite{erdos-remarks} has been kindly provided to us (via the author's personal blog) by Shengtong Zhang, which we reproduce below.  Given this inequality, the arguments in \Cref{main-sec} can be modified as follows in order to obtain the analogue of \Cref{thm:main} for the sum of divisors function:
\begin{itemize}
    \item All occurrences of $\varphi$ should now be replaced with $\sigma$ (including in the definition of $M()$).
    \item In the proof of \Cref{a5}, all displayed expressions of the form $1 - \frac{\dots}{p}$ should be replaced with $1 + \frac{\dots}{p}$ (and similarly with $p$ replaced by $p'$), and ``$I_{i,k'}$ lies strictly to the right of $I_{i,k}$'' should be replaced with ``$I_{i,k'}$ lies strictly to the left of $I_{i,k}$''.
    \item In the proof of \Cref{a7}, all displayed expressions of the form $1 - \frac{1}{p}$ or $1 - O(\frac{1}{D^3})$ should be replaced with $1 + \frac{1}{p}$ or $1 + O(\frac{1}{D^3})$ respectively, and \Cref{prelim} must be replaced with \eqref{pol}.
\end{itemize}
In particular, the largest subset of $[x]$ on which $\sigma$ is nondecreasing has cardinality $\left(1+O(\frac{\log_2^5 x}{\log x})\right) \pi(x)$.

Now we present Zhang's proof of \eqref{pol}.  We may assume that $q \neq 1$, since the only solution to $\frac{\sigma(d)}{d}=1$ is $d=1$.
We split any natural number $d$ with $\frac{\sigma(d)}{d}=q$ uniquely into a product $d = \gamma s$ of a powerful number $\gamma$ (a number which is not divisible precisely once by any prime) and a squarefree number $s$ coprime to $\gamma$.  The key observation (a minor variant of the lemma in \cite[p. 173]{erdos-remarks}) is that for fixed $q$, the squarefree number $s$ is determined by $\gamma$.  Indeed, suppose we had two numbers $d = \gamma s$, $d' = \gamma s'$ with $\frac{\sigma(d)}{d} = \frac{\sigma(d')}{d'} = q$ with $s \neq s'$ squarefree and coprime to $\gamma$.  Then by multiplicativity we have
$$ \frac{\sigma(s)}{s} = \frac{\sigma(s')}{s'}.$$
If we let $a \coloneqq s/(s,s')$, $a' = s'/(s,s')$, then $a,a'$ are coprime and square-free with
\begin{equation}\label{aap}
 \frac{\sigma(a)}{a} = \frac{\sigma(a')}{a'}.
\end{equation}
If $a=1$ then $a'=1$ and hence $s=s'$, contradiction, and similarly if $a'=1$. Thus we may assume $a,a' > 1$.  At least one of $a,a'$ is not divisible by $3$, so without loss of generality we may assume that $a$ is not divisible by $3$.  Then the largest prime factor of $a$ divides $a$ but not $a'$ or $\sigma(a)$, contradicting \eqref{aap}, and the key observation follows.

We may thus write
$$ \sum_{\frac{\sigma(d)}{d}=q} \frac{1}{d} = \sum_\gamma \frac{1}{\gamma s_\gamma}$$
where $\gamma$ ranges over powerful numbers, and $s_\gamma = s_{\gamma,q}$ is equal to the (necessarily unique) square-free number coprime to $\gamma$ with $\frac{\sigma(\gamma s_\gamma)}{\gamma s_\gamma} = q$, or $s_\gamma \coloneqq +\infty$ if no such number exists.  Since every powerful number is the product of a square and a cube, we have
$$ \sum_\gamma \frac{1}{\gamma} \leq \zeta(2) \zeta(3) = 1.97730\dots < 2$$
(in fact the sum is about $1.94359\dots$).  This already establishes \eqref{pol} but with an upper bound of $2$ instead of $1$.  To refine the upper bound to $1$, it will suffice to establish the lower bound
$$ \sum_\gamma \frac{1}{\gamma} \left(1 - \frac{1}{s_\gamma}\right) \geq 1.$$
Since the summand is non-negative, it will suffice to establish this bound for the powers of $2$, thus we have reduced to showing that
\begin{equation}\label{final}
    1 - \frac{1}{s_1} + \sum_{j=2}^\infty \frac{1}{2^j} \left(1 - \frac{1}{s_{2^j}}\right) \geq 1.
\end{equation}

This can be verified by the following case analysis.
\begin{itemize}
    \item  The case $s_1=1$ cannot occur since $q$ is assumed to not equal $1$.
    \item  If $s_1=2$, then $q = \frac{\sigma(2)}{2}=\frac{3}{2}$, but then $s_{2^j}=\infty$ for $j \geq 2$ since $\frac{\sigma(2^j s)}{2^j s} \geq \frac{2^{j+1}-1}{2^j} > \frac{3}{2}$ for any $j \geq 2$ and $s$ coprime to $j$, and so the left-hand side of \eqref{final} is $1 - \frac{1}{2} + \sum_{j=2}^\infty \frac{1}{2^j} = 1$.
    \item If $s_1=\infty$, then the left-hand side of \eqref{final} is clearly at least $1 - \frac{1}{s_1}=1$.
    \item Finally, if $3 \leq s_1 < \infty$, then $q = \frac{\sigma(s_1)}{s_1}$ has a squarefree denominator, which implies for any $j \geq 2$ and $s$ coprime to $2^j$ that $\frac{\sigma(2^j)}{2^j} = \frac{2^{j+1}-1}{2^j}$ is not equal to $q$ (the denominator is divisible by $2^2$).  Hence $s_{2^j}$ cannot equal $1$, and must therefore be at least $3$ since $s_{2^j}$ must be coprime to $2^j$.  Thus the left-hand side of \eqref{final} is at least $1 - \frac{1}{3} + \sum_{j=2}^\infty \frac{1}{2^j} (1-\frac{1}{3}) = 1$.
\end{itemize}

\begin{remark}
    The argument in fact shows that equality in \eqref{pol} can only occur when $q=1$, so in some sense the situation with $\sigma$ is slightly better than that with $\varphi$ (in which $q=1/2$ is also an equality case). A variant of the above argument also shows that for any $\eps>0$, there are only finitely many $q$ for which the sum $\sum_{\frac{\sigma(d)}{d}=q} \frac{1}{d} \geq \eps$ (because one must have $s_\gamma \ll \frac{1}{\eps}$ for some powerful $\gamma \ll_\eps 1$).
\end{remark}

\begin{remark}\label{dedekind-rem}  Analogous results also hold when $\phi$ or $\sigma$ is replaced by\footnote{We thank an anonymous referee for this observation.} the Dedekind totient function
$$ \psi(n) \coloneqq n \prod_{p|n} \frac{p+1}{p}.$$
Again, the main difficulty is to establish an analogue of Proposition \ref{prelim} for this function.  The proof of Lemma \ref{exact} almost works for $\psi$ with only obvious modifications; however, due to the presence of the consecutive primes $2, 3$, there is one exception to the assertion that when
$$ \frac{\psi(d)}{d} = \prod_{p|d} \frac{p+1}{p}$$
is expressed in lowest terms as $\frac{a}{b}$, that the largest prime factor of $b$ is that of $d$.  Namely, when the prime factors of $d$ are precisely $2$ and $3$, then there is a cancellation
$$ \frac{\psi(d)}{d} = \frac{2+1}{2} \frac{3+1}{3} = \frac{2}{1}$$
and the largest prime factor of $b$ is now (by our convention) $1$ instead of $3$.  However, because $2,3$ are the only pair of consecutive primes, this is the unique exception to the above assertion.  As a consequence, it is possible to modify the inductive argument used to establish Lemma \ref{exact} to prove an analogue for $\psi$, with the one modification that if $q = \frac{2}{1}$, then the largest element of ${\mathcal P}$ is equal to $3$ rather than ${\mathcal P}$ being empty.  (For this, it is convenient to take the base case to be the case when the largest prime factor of $b$ is at most $3$, which can be handled by a finite case analysis.)  The rest of the arguments then continue as before with only obvious modifications; we leave the details to the interested reader.
\end{remark}

\end{document}